\documentclass[letterpaper, 10 pt, conference]{ieeeconf}  

\IEEEoverridecommandlockouts                              
\usepackage{bbm}
\usepackage{mathrsfs}
\usepackage{verbatim}
\usepackage{amsmath}
\usepackage{amsfonts}
\usepackage{bm}
\usepackage{graphicx}
\usepackage{caption}
\usepackage{float}

\usepackage{amsthm}
\usepackage{xcolor}
\usepackage[colorinlistoftodos]{todonotes}
\usepackage{mathtools}
\usepackage{cite}

\theoremstyle{definition}

\newtheorem{definition}{Definition}
\newtheorem{remark}{Remark}
\newtheorem{problem}{Problem}

\theoremstyle{plain}

\newtheorem{theorem}{Theorem}
\newtheorem{lemma}{Lemma}

\title{\LARGE \bf
A Dynamic Program for a Team of Two Agents with Nested Information}

\author{Aditya Dave, {\itshape{Student Member, IEEE,}} and Andreas A. Malikopoulos, {\itshape{Senior Member, IEEE}} 
	\thanks{This research was supported by the Sociotechnical Systems Center (SSC) at the University of Delaware.} %
	\thanks{The authors are with the Department of Mechanical Engineering, University of Delaware, Newark, DE 19716 USA (email: \texttt{adidave@udel.edu; andreas@udel.edu).}} }

\begin{document}

\maketitle
\thispagestyle{empty}

\begin{abstract}
In this paper, we investigate a sequential dynamic team problem consisting of two agents with a nested information structure. We use a combination of the person-by-person and prescription approach to derive structural results for optimal control strategies for the team. We then use these structural results to present a dynamic programming (DP) decomposition to derive the optimal control strategies for a finite time horizon. We show that our DP utilizes the nested information structure to simplify the computation of the optimal control laws for the team at the final time step.
\end{abstract}

\section{Introduction}

Team theory refers to problems where a team of agents seeks to cooperatively control a state and minimize a shared cost \cite{radner1962team}, with applications including 
connected and automated vehicles \cite{malikopoulos2021optimal}, social media platforms \cite{Dave2020SocialMedia}, and robot swarms \cite{Beaver2020AnFlockingb}.
A key aspect of these problems is the team's \textit{information structure}, which describes the information available to each agent at any time. Various information structures are categorized as:
(1) \textit{Classical:} Each agent receives the same information and has perfect recall \cite{varaiya_book}.
(2) \textit{Quasi-classical:} If agent $1$ can affect the information of agent $2$, the information available to agent $1$ is \textit{also} available to agent $2$ \cite{2, lessard2013structural, nayyar2015structural, yuksel2009stochastic, Aditya2019}.
(3) \textit{Non-classical:} All other decentralized information structures are called non-classical \cite{17, 19, dave2020structural, malikopoulos2021team}.

In this paper, we analyze a dynamic team of two agents with a \textit{nested information structure}. In the nested information structure, agent $2$ shares her information with agent $1$ at each instance of time, but does not receive any information from agent $1$. Both agents collectively control and partially observe a shared state. Thus, this is a non-classical information structure. The nested information structure is commonly found in applications with real time communication problems \cite{nayyar2011structure}, vehicle platoons \cite{mahbub2021_platoonMixed}, and hierarchical control problems \cite{mahajan2015algorithmic}. 
A team of two agents with a quasi-classical, partially nested information structure, is also special case of the nested team when agent $2$ can affect the state of agent $1$, for example \cite{nayyar2015structural}. The two agent team is of interest because most insights into the structure of optimal control strategies in a team of two agents can be extended teams of many agents with a similar information structure.
Such partially nested teams are well understood in the literature for linear dynamics, quadratic costs, and Gaussian noise (LQG) \cite{2,lessard2013structural, nayyar2015structural} and for nonlinear dynamics with complete state observation \cite{mahajan2015algorithmic}. Several dynamic program (DP) decompositions have been reported in the literature for decentralized teams \cite{nayyar2019common, yuksel2020universal}, which are reviewed in detail in \cite{mahajan2012, malikopoulos2021team}. Furthermore, a DP that can be applied in team problems with nested information structures is the one proposed in \textit{the common information approach} \cite{17, nayyar2019common}, which introduces a coordinator who selects prescription functions for each agent. 




Our main contribution in this paper is to present structural results for optimal control strategies in the nested information structure which cannot be derived solely using the common information approach. Our analysis uses a combination of \textit{the person-by-person approach} \cite{2, lessard2013structural, nayyar2015structural, charalambous2016decentralized} and \textit{the prescription approach} \cite{dave2020structural}. Similar techniques have been used in conjunction with linear dynamics \cite{gagrani2017decentralized, nayyar2015structural}, where the dynamics ensure that the optimal control strategies depend only on the expected value of certain random variables. This ensures the tractability of the eventual DP. For nonlinear dynamics, similar techniques have been used for real time communication \cite{nayyar2011structure} and control sharing information structures \cite{mahajan2013controlsharing} by assuming specific dynamics. However, when agents imperfectly observe their states, the optimal control strategies derived in these papers are functions of non-parametric probability distribution with a continuous support. Thus, it is challenging to actually implement these optimal strategies. In contrast, we derive optimal control strategies without assuming specific dynamics and our strategies only require tracking probability distributions with an atomic support. However, our structural results yield strategies whose domain grows in size with time and thus, can only be applied to finite time horizons. Despite this, we believe that our results may be useful in the search for approximately optimal control strategies \cite{subramanian2019approximate} that may be time invariant. We also present a DP based on our structural results. Our DP deviates from other DPs in the literature for the final time step, where we utilize the the nested information structure to improve the computational efficiency of optimal control strategies.


The remainder of the paper proceeds as follows. In Section II, we provide the problem formulation. In Section III, we analyze a team of two agents using the person-by-person and prescription approaches, and derive structural results for optimal control strategies. In Section IV, we present a DP to derive the optimal control strategies. Finally, in Section V, we present concluding remarks and discuss ongoing work.

\section{Problem Formulation}

We consider a team of two agents who select actions over $T \in \mathbb{N}$ discrete time steps.
At any time $t=0,\dots, T$, the state of the team is denoted by the random variable $X_t$
that takes values in a finite set of feasible states $\mathcal{X}_t$.
The control action of agent $k \in \{1,2\}$ at each time $t$ is denoted by the random variable $U_t^k$ that takes values in a finite set of feasible actions $\mathcal{U}_t^k$.
Let ${U}_t^{1:2}=(U_t^1,U_t^2)$. 
Starting with the initial state $X_0$ at $t=0$, the system evolves as
\begin{equation}
    X_{t+1}=f_t\left(X_t,U_t^{1:2},W_t\right), \quad t =0,\dots,T-1, \label{st_eq}
\end{equation}
where the random variable $W_t$ denotes an uncontrolled disturbance to the state at time $t$ and takes values in a finite set of feasible disturbances $\mathcal{W}_t$.
At each time $t = 0,\dots,T$, each agent $k \in \{1,2\}$ partially observes the state as a random variable $Y_t^k$ that takes values in a finite set $\mathcal{Y}^k_t$. The observation $Y_t^k$ is given by
\begin{equation}
    Y_t^k=h_t^k(X_t,V_t^k), \quad t=0,\dots,T, \label{ob_eq}
\end{equation}
where the random variable $V_t^k$ denotes the measurement noise that takes values in a finite set of feasible noises $\mathcal{V}^k_t$.
The external disturbances $\{W_t: t=0,\dots,T\}$, noises in measurement $\{V_t^1, V_t^2: t=0,\dots,T\}$, and initial state $X_0$ are collectively called the \textit{primitive random variables} of the team and their probability distributions are known a priori.
We assume that each primitive random variable is independent of all other primitive random variables. This ensures that the state $X_t$ evolves as a controlled Markov chain at each $t=0,\dots,T$ \cite{varaiya_book}.
Next, we define the \textit{memory} of each agent $k \in \{1,2\}$ at each time $t$, which is a collection of all the data received by them. The memories of both agents determine the team's \textit{information structure}.

\begin{definition}
The memory of agent $k \in \{1,2\}$ at each time $t = 0,\dots,T$ is a set of random variables $M_t^k$ 
that takes values in a finite collection of sets $\mathcal{M}^k_t$. 
\end{definition}

Let $U_{0:t}^k=(U_0^k,\dots,U_t^k)$, $k \in \{1,2\}$.
The memories of the two agents in the team at any time $t = 0,\dots,T$ are given by
\begin{align}
    M_t^2 &:= \{Y_{0:t}^2, U_{0:t-1}^2\}, \label{m2_def} \\
    M_t^1 &:= \{Y_{0:t}^1, U_{0:t-1}^1, Y_{0:t}^2, U_{0:t-1}^2\}. \label{m1_def}
\end{align}
In \eqref{m2_def}-\eqref{m1_def}, we consider that each agent updates her memory before generating her action at each time $t$.
Note that the memories satisfy the following properties for all $t=0,\dots,T$:
(1) \textit{causality:}
    $M_t^k \subseteq \{Y_{0:t}^{1:2},U_{0:t-1}^{1:2}\}$, $k \in \{1,2\}$;
(2) \textit{perfect recall:} 
$M_t^k \subseteq M_{t+1}^k$, $k \in \{1,2\}$; and
(3) \textit{nested information structure:}
$M_t^2 \subseteq M_t^1$.

\begin{remark}
Causality and perfect recall are very general properties used to derive DPs for both centralized \cite{varaiya_book} and decentralized \cite{mahajan2012} teams. The nested information structure is unique to our team.
\end{remark}

The new information of any agent $k$ at each $t$ is the set
    $Z_t^k := M_t^k \setminus M_{t-1}^k$
that takes values in a finite collection of sets $\mathcal{Z}_t^k$. 
Thus,
\begin{align}
    Z_t^1 &= \{Y_t^{1}, U_{t-1}^{1}, Y_{t}^{2}, U_{t-1}^{2}\}, \quad t = 0,\dots,T, \label{z_1} \\
    Z_t^2 &= \{Y_{t}^{2}, U_{t-1}^{2}\}, \quad \quad \quad \quad \quad \; t = 0,\dots,T, \label{z_2}
\end{align}
which implies that
\begin{gather}
    Z_t^2 \subset Z_t^1, \quad  t=0,\dots,T. \label{z_relation}
\end{gather}

Each agent $k \in \{1,2\}$ at each $t = 0,\dots,T$ selects an action $U_t^k$ as a function of her memory $M_t^k$. Thus,
\begin{gather} \label{U_def}
    U_t^k = g_t^k(M_t^k), \quad t=0,\dots,T,
\end{gather}
where $g_t^k$ is the \textit{control law} of agent $k$ at time $t$. The \textit{control strategy} of each agent $k$ is $\boldsymbol{g}^k := (g_{0}^k,\dots,g_T^k)$ and the \textit{strategy profile} of the team is $\boldsymbol{g} := (\boldsymbol{g}^1,\boldsymbol{g}^2)$. The set of all feasible strategy profiles is $\mathcal{G}$. 
After each $k \in \{1,2\}$ selects action $U_t^k$ at time $t$, the team incurs a cost $c_t(X_t,U_t^1,U_t^2) \in \mathbb{R}_{\geq0}$. Then, the performance criterion for the system is 
\begin{equation}
    \mathcal{J}(\boldsymbol{g}) = \mathbb{E}^{\boldsymbol{g}}\left[\sum_{t=0}^T{c_t\big(X_t,U_t^{1},U_t^2\big)}\right], \label{per_cri}
\end{equation}
where the expectation is with respect to the joint distribution on all random variables, and $U_t^k$ is given by \eqref{U_def} for each agent $k \in \{1,2\}$, at each time $t = 0,\dots,T$. Then, we can state the optimization problem for the team as follows.

\begin{problem} \label{problem_1}
The optimization problem is $\inf_{\boldsymbol{g} \in \mathcal{G}} \mathcal{J}(\boldsymbol{g}),$
given the probability distributions of the primitive random variables $\{X_0,W_{t},V_{t}^1,V_{t}^2:t = 0,\dots,T\}$, and the functions $\left\{c_t,f_t,h_t^{1:2}:t=0,\dots,T\right\}$.
\end{problem}

Our aim is to develop a DP that can tractably derive an optimal strategy profile $\boldsymbol{g}^* \in \mathcal{G}$ for Problem \ref{problem_1}, such that $\mathcal{J}(\boldsymbol{g}^*) \leq \mathcal{J}(\boldsymbol{g})$, for all $\boldsymbol{g} \in \mathcal{G}$.

\section{Analysis}

\subsection{The Person-by-Person Approach}

In this subsection, we present a structural result for the optimal control strategy of agent $1$ using the person-by-person approach. This will help us derive our DP in Section IV. We first fix a control strategy $\boldsymbol{g}^2$ for agent $2$, such that
\begin{gather}
    U_t^2 = g_t^2(M_t^2), \quad t = 0,\dots,T. \label{U_2_fixed}
\end{gather}
In this approach, given the strategy $\boldsymbol{g}^2$ of agent $2$, we set up a centralized problem from the perspective of agent $1$. Since $M_t^2 \subseteq M_t^1$ at each time $t = 0,\dots,T$, given the control strategy $\boldsymbol{g}^2$, agent $1$ can derive the action $U_t^2$ using \eqref{U_2_fixed}.
Then, we can define a new state for agent $1$ as
\begin{gather} \label{eq_st_1}
    S_t^1 := \{X_t, M_t^2\}, \quad  t = 0,\dots,T,
\end{gather}
that takes values in a finite collection of sets $\mathcal{S}_t^1$ 
at any time $t$. Next, we show that 
the new state 
is sufficient for input-output mapping. 

\begin{lemma} \label{lem_pbp_suff}
Let $\boldsymbol{g}^2$ be a given control strategy of agent $2$. At each time $t=0,\dots,T$, the state $S_t^1 \in \mathcal{S}_t^1$ is sufficient for input-output mapping by the following properties \cite{witsenhausen1976some}:

1) There exist functions $\hat{f}^{{1}}_t(\cdot)$ and $\hat{h}^{{1}}_t(\cdot)$  for all $t=0,\dots,T-1$,, such that
  \begin{align}
	S^{{1}}_{t+1} &= \hat{f}^{{1}}_{t}(S^{{1}}_t, U_t^{{1}}, W_t, V_{t+1}^{1:2}), \label{eq_St_pbp_1} \\
	Z^{{1}}_{t+1} &= \hat{h}^{{1}}_{t}(S^{{1}}_{t},U_t^{{1}},W_t,V_{t+1}^{1:2}). \label{eq_St_pbp_2}
  \end{align}
  
2) There exist functions $\hat{c}^{{1}}_t(\cdot)$, such that
    \begin{align}
	c_t(X_t,U_t^{1},U_t^{2}) &= \hat{c}^{{1}}_t(S^{{1}}_t,U^{{1}}_t), \; \; \;  t = 0,\dots,T. \label{eq_St_pbp_3}
  \end{align}
\end{lemma}

\begin{proof}
To prove these results, we expand the LHS in each of \eqref{eq_St_pbp_1}-\eqref{eq_St_pbp_3} by substituting appropriate relations from the system dynamics \eqref{st_eq}, \eqref{ob_eq}, the definitions \eqref{z_1}, \eqref{eq_st_1}, and \eqref{U_2_fixed}, for each $t=0,\dots,T$. Thus, we can rewrite the LHS in terms of the variables in the RHS and construct appropriate functions $\hat{f}_t^1(\cdot)$, $\hat{h}_t^1(\cdot)$, and $\hat{c}_t^1(\cdot)$ for each time $t$.
\end{proof}

Given the strategy $\boldsymbol{g}^2$, Lemma \ref{lem_pbp_suff} leads to a centralized problem for agent $1$, with state $S_t^1$, control action $U_t^1$, observation $Z_{t+1}^1$, and cost $\hat{c}_t^1(S_t^1,U_t^1)$ at each time $t = 0,\dots,T$. The performance criterion is solely a function of the control strategy $\boldsymbol{g}^1$, as 
    $\mathcal{J}^1(\boldsymbol{g}^1) = \mathbb{E}^{\boldsymbol{g}} \left[\sum_{t=0}^T \hat{c}^1_t(S_t^1, U_t^1)\right]$,
where the expectation is with respect to the joint probability distribution on all random variables and $\boldsymbol{g}=(\boldsymbol{g}^1,\boldsymbol{g}^2)$.

\begin{problem} \label{problem_2}
The problem for agent $1$ is
    $\inf_{\boldsymbol{g}^1} \mathcal{J}^1(\boldsymbol{g}^1)$,
given the control strategy $\boldsymbol{g}^2$, the probability distributions of the primitive random variables $\{X_0,W_{t},V_{t}^1,V_{t}^2:t = 0,\dots,T\}$, and the functions $\left\{c_t,f_t,h_t^{1:2}:t=0,\dots,T\right\}$.
\end{problem}

In Problem \ref{problem_2}, at each time $t$, the component $M_t^2$ of the state $S_t^1$ is observed by agent $1$. However, the component $X_t$ of state $S_t^1$ must be inferred by agent $1$ using her memory $M_t^1$. For such a problem, it is known \cite[page 79]{varaiya_book} that agent $1$ can estimate $X_t$ using the probability distribution 
\begin{gather} \label{pi_1_def}
    \Pi_t^1 := \mathbb{P}^{\boldsymbol{g}}\big(X_t ~|~ M_t^1\big), \quad t=0,\dots,T,
\end{gather}
that takes values in the set of feasible distributions $\mathcal{P}^1_t := \Delta(\mathcal{X}_t)$ at each time $t$. The distribution $\Pi_t^1$ is called an \textit{information state} for agent $1$ and yields the following structural result for the control strategy of agent $1$ in Problem \ref{problem_2}.

\begin{theorem} \label{pbp_struct_result}
    Let $\boldsymbol{g}^2$ be a given control strategy for agent $2$. Then, the optimal control strategy $\boldsymbol{g}^{*1}$ of agent $1$ in Problem \ref{problem_2} has the structural form
    \begin{gather} \label{eq_pbp_struct_result}
        U_t^1 = g_t^{*1}(M_t^2, \Pi_t^1), \quad t = 0,\dots,T.
    \end{gather}
\end{theorem}

\begin{proof}
This result follows from standard arguments for partially observed Markov decision processes \cite[page 79]{varaiya_book}.
\end{proof}

Note that every optimal strategy profile $\boldsymbol{g}^* = (\boldsymbol{g}^{*1}, \boldsymbol{g}^{*2})$ for Problem \ref{problem_1}, must be a solution of Problem \ref{problem_2} by fixing $\boldsymbol{g}^{*2}$ for agent $2$ and selecting the control strategy of agent $1$ as $\boldsymbol{g}^{*1}$ \cite{2}. Thus, every optimal profile $\boldsymbol{g}^*$ for Problem \ref{problem_1} also satisfies Theorem \ref{pbp_struct_result}. Then, in Problem \ref{problem_1}, we can restrict our attention to strategy profiles $\boldsymbol{g} \in \mathcal{G}$ with the structural form
\begin{align}
    U_t^1 &= g_t^1(M_t^2, \Pi_t^1), \label{u_pbp} \\
    U_t^2 &= g_t^2(M_t^2), \label{u_2}
\end{align}
at each time $t=0,\dots,T$. To this end, we denote the set of feasible strategy profiles consistent with \eqref{u_pbp}-\eqref{u_2} by $\mathcal{G}'$. 

\subsection{The Prescription Approach}

In this subsection, we consider Problem 1 with the restriction $\boldsymbol{g} \in \mathcal{G}'$.
Any strategy profile $\boldsymbol{g} \in \mathcal{G}'$ for the team is accessible to both agents. However, at any time $t=0,\dots,T$, agent $2$ cannot generate the action $U_t^1$ using \eqref{u_pbp}, because she can only access the memory $M_t^2$ and not the information state $\Pi_t^1 \in \mathcal{P}_t^1$, which is a function of the memory $M_t^1$. Instead, agent $2$ considers that the action $U_t^1$ is generated in two stages at each time $t$: (1) agent $1$ generates a function using only $M_t^2$, and (2) this function takes as an input the information state $\Pi_t^1$ to generate the action $U_t^1$. We call this function a \textit{prescription} of agent $2$ for agent $1$ at time $t$. 

\begin{definition}
A \textit{prescription} of agent $2$ for agent $1$ at any time $t=0,\dots,T$ is a function $\Gamma_t^{[2,1]}: \mathcal{P}^{1}_t \to \mathcal{U}_t^1$ that takes values in a finite set of feasible functions $\mathcal{F}^{[1,2]}_t$. 
\end{definition}

The prescription $\Gamma_t^{[2,1]}$ is generated as
\begin{gather} \label{gamma}
    \Gamma_t^{[2,1]} = \psi_t^{[2,1]}(M_t^2), \quad t=0,\dots,T,
\end{gather}
where $\psi_t^{[2,1]}: \mathcal{M}_t^2 \to \mathcal{F}_t^{[2,1]}$ is called the prescription law of agent $2$ for agent $1$ at time $t$. We call $\boldsymbol{\psi}^2 := \big(\psi_t^{[2,1]} : t=0,\dots,T\big)$ the prescription strategy of agent $2$, and denote the set of feasible prescription strategies by $\Psi^2$. Next, Lemmas \ref{lem_psi_g_relation} and \ref{lem_psi_g_relation_inv} show that any control action $U_t^1$ can be equivalently generated using either a control strategy $\boldsymbol{g} \in \mathcal{G}'$ or an appropriate prescription strategy $\boldsymbol{\psi}^2$.

\begin{lemma} \label{lem_psi_g_relation}
For any given control strategy $\boldsymbol{g} \in \mathcal{G}'$, we can construct a prescription strategy $\boldsymbol{\psi}^2 \in \Psi^2$ such that 
\begin{gather}
    \Gamma_t^{[2,1]}\left(\Pi_t^1\right) = g_t^1(M_t^2,\Pi_t^1) = U_t^1, \quad t =0,\dots,T. \label{u_presc}
\end{gather}
\end{lemma}

\begin{proof}
For any control law ${g}_t^1$ that generates $U_t^1$ at time $t$ using \eqref{u_pbp}, we can construct a prescription law $\psi_t^{{[2,1]}} : \mathcal{M}_t^2 \to \mathscr{F}_t^{[2,1]}$ as 
$\Gamma_t^{[2,1]}(\Pi^1_t) = \psi_t^{[2,1]}(M_t^2)(\Pi_t^1) := g_t^1(M_t^2,\Pi_t^1) = U_t^1$, for all $t = 0,\dots,T.$
\end{proof}

\begin{lemma}
\label{lem_psi_g_relation_inv}
For any given prescription strategy $\boldsymbol{\psi}^{{2}} \in \Psi^{{2}}$, we can construct a control strategy $\boldsymbol{g} \in \mathcal{G}'$ such that
\begin{gather}
     g_t^{{1}}(M_t^2, \Pi^1_t) = \Gamma_t^{[2,1]}(\Pi_t^1) = U_t^{{1}}, \quad t =0,\dots,T. \label{u_presc_inv}
\end{gather}
\end{lemma}

\begin{proof}
For any prescription strategy $\boldsymbol{\psi}^{{2}},$ we construct a control strategy $\boldsymbol{g}$ such that
$g_t^{{1}}(M_t^2, \Pi^1_t) := {\psi}_t^{[2,1]}(M_t^2)(\Pi^1_t) = \Gamma_t^{[2,1]}(\Pi_t^1) = U_t^1,$ for all $t=0,\dots,T$.
\end{proof}

Lemmas \ref{lem_psi_g_relation} and \ref{lem_psi_g_relation_inv} imply that every control action $U_t^{{1}}$ of an agent ${{1}}$ generated through a control strategy $\boldsymbol{g}^1$ can also be generated through an appropriate prescription strategy $\boldsymbol{\psi}^{{2}}$ and vice versa. We always choose control and prescription strategies that satisfy \eqref{u_presc} and \eqref{u_presc_inv}. Thus, at each time $t$,
\begin{align}
    \Pi_t^1 &= \mathbb{P}^{\boldsymbol{g}} (X_t ~|~ M_t^1) = \mathbb{P}^{(\boldsymbol{g}, \boldsymbol{\psi}^2)} (X_t ~|~ M_t^1, \Gamma_{0:t-1}^{[2,1]}), \label{pi_1_with_gamma}
\end{align}
where we use Lemma \ref{lem_psi_g_relation} to construct $\boldsymbol{\psi}^2$ given $\boldsymbol{g}$, and we can add the history of prescriptions $\Gamma_{0:t-1}^{[2,1]}$ to the conditioning because they are simply functions of $M_t^2 \subseteq M_t^1$ and $\boldsymbol{\psi}^2$. 
Using similar arguments, it holds that
    $\mathcal{J}(\boldsymbol{g}) 
    = \mathbb{E}^{(\boldsymbol{g},\boldsymbol{\psi}^2)}\left[\sum_{t=0}^T c_t(X_t,U_t^{1:2})\right]$,
where $U_t^1$ is equivalently generated using either $g_t^1(M_t^2,\Pi_t^1)$ or $\Gamma_t^{[2,1]}(\Pi_t^1)$.

\subsection{A New State for Agent $2$}

In this subsection, we define a state sufficient for input-output mapping for agent $2$. 
We first define the information accessible to agent $1$ but \textit{inaccessible} to agent $2$ at each time $t=0,\dots,T$ as the set of random variables
\begin{gather}
    L_t^{[1,2]} := M_t^1 \setminus M_t^2,
\end{gather}
that takes values in a finite collection of sets $\mathcal{L}^{[1,2]}_t$. For all $t$, we define an information state for agent $2$ as the distribution
\begin{gather} \label{pi_2_def}
    \Pi_t^2 := \mathbb{P}^{(\boldsymbol{g}, \boldsymbol{\psi}^2)}\big(X_t, L_t^{[1,2]}~|~M_t^2, \Gamma_{0:t-1}^{[2,1]}\big),
\end{gather}
that takes values in the set of feasible distributions $\mathcal{P}^2_t := \Delta(\mathcal{X}_t \times \mathcal{L}_t^{[1,2]})$. Next, we show that we can write the information state $\Pi_t^1$ of agent $1$ in terms of $\Pi_t^2$ at each $t=0,\dots,T$.


\begin{lemma} \label{pi_rel}
At any time $t=0,\dots,T$, 
for the pair of probability distributions $\Pi_t^1$ and $\Pi_t^2$, we can construct a function ${e}_t: \mathcal{P}_t^2 \times \mathcal{L}_t^{[1,2]}  \to  \mathcal{P}_t^1$, such that
\begin{gather}
    \Pi_t^1 = {e}_t\big(\Pi_t^2, L_t^{[1,2]} \big). \label{eq_pi_rel}
\end{gather}
\end{lemma}

\begin{proof}
The proof is omitted due to space constraints, but can be found in our online preprint \cite[Appendix A]{dave2021dynamic}.
\end{proof}

Thus, at each time $t$, we can equivalently write the control action as
    $U_t^1 = \Gamma_t^{[2,1]}(\Pi_t^1) = \Gamma_t^{[2,1]}\big(e_t(\Pi_t^2,L_t^{[1,2]})\big).$
Next, we construct a new state for agent $2$ as
\begin{gather}
    S_t^2 := \big\{X_t, L_t^{[1,2]}, \Pi_t^2\big\}, \quad t = 0, \dots, T,
\end{gather}
that takes values in the finite collection of sets ${\mathcal{S}}_t^2$. Our goal is to set up an equivalent centralized control problem for agent $2$ with the state $S_t^2$ and control action $(\Gamma_t^{[2,1]},U_t^2)$ at each time $t$. However, we require some interim results before we can prove that the state $S_t^2$ is sufficient for input-output mapping.
Next, we show that the information states $\Pi_t^2$ and $\Pi_t^1$ at all $t$ are independent from the strategies $(\boldsymbol{g},\boldsymbol{\psi}^2)$.

\begin{lemma} \label{pi_2_1}
At each time $t = 0,\dots,T-1$, there exists a function $\tilde{f}_t^{{2}}(\cdot)$ independent from $(\boldsymbol{g},\boldsymbol{\psi}^2)$, such that
\begin{equation} \label{pi_2_1_relation}
    \Pi_{t+1}^{{2}} = \Tilde{f}_{t}^{{2}}(\Pi_t^{{2}},\Gamma_t^{{[2,1]}},U_t^2,Z_{t+1}^{{2}}),
\end{equation}
and subsequently, for any Borel subset $P^2 \subset \mathcal{P}^2_{t+1}$,
\begin{equation*}
    \mathbb{P}(\Pi_{t+1}^{{2}} \in P^2|~M_t^2, \Gamma_{0:t}^{[2,1]}) = \mathbb{P}(\Pi_{t+1}^{{2}} \in P^2|~\Pi_t^2, U_t^2, \Gamma_{t}^{[2,1]}).
\end{equation*}
\end{lemma}

\begin{proof}
The proof is omitted due to space constraints, but can be found in our online preprint \cite[Appendix B]{dave2021dynamic}.
\end{proof}

\begin{lemma} \label{pi_1_1}
At each time $t = 0,\dots,T-1$, there exists a function $\tilde{f}_t^{{1}}(\cdot)$ independent from $(\boldsymbol{g}, \boldsymbol{\psi}^2)$, such that
\begin{equation} \label{pi_1_1_relation}
    \Pi_{t+1}^{{1}} = \Tilde{f}_{t}^{{1}}(\Pi_t^{{1}},U_t^{{1}}, U_t^2, Z_{t+1}^{{1}}),
\end{equation}
and subsequently, for any Borel subset $P^1 \subset \mathcal{P}^1_{t+1}$,
\begin{equation*}
    \mathbb{P}(\Pi_{t+1}^{{1}} \in P^1|~M_t^1, U_t^2) = \mathbb{P}(\Pi_{t+1}^{{1}} \in P^1|~\Pi_t^1, U_t^1, U_t^2).
\end{equation*}
\end{lemma}

\begin{proof}
The proof is similar to Lemma \ref{pi_2_1} and is omitted.
\end{proof}



\begin{lemma} \label{lem_pi_cost}
At each time $t = 0,\dots,T$, there exists a function $\tilde{c}_t^k(\cdot)$ for each $k \in \{1,2\}$ such that
\begin{align}
    \mathbb{E}^{\boldsymbol{g}}[c_t(X_t,U_t^{1:2}) | M_t^1,U_t^{1:2}] &= \tilde{c}_t^1(\Pi_t^1,U_t^{1:2}), \label{eq_lem_pi_cost_1} \\
    \mathbb{E}^{\boldsymbol{g}}[\tilde{c}_t^1(\Pi_t^1,U_t^{1:2}) | M_t^2,\Gamma_t^{[2,1]},U_t^{2}] &= \tilde{c}_t^2(\Pi_t^2,\Gamma_t^{[2,1]},U_t^{2}). \label{eq_lem_pi_cost_2}
\end{align}
\end{lemma}

\begin{proof}
We first prove \eqref{eq_lem_pi_cost_1}. Let $m_t^1$, $u_t^{1:2}$, and $\pi_t^1$ be the realizations of the random variables $M_t^1$, $U_t^{1:2}$, and the conditional distribution $\Pi_t^1$ at each time $t = 0,\dots,T$. Then, we expand the expectation as
$\mathbb{E}^{\boldsymbol{g}}[c_t(X_t,U_t^{1:2}) ~|~ m_t^1, u_t^{1:2}]
    = \sum_{x_t} c_t(x_t,u_t^{1:2}) \cdot \mathbb{P}^{\boldsymbol{g}}(X_t = x_t ~|~ m_t^1,u_t^{1:2})
    = \sum_{x_t} c_t(x_t,u_t^{1:2}) \cdot \pi_t^1(x_t) =: \tilde{c}_t^1(\pi_t^1,u_t^{1:2}),$
where we can drop the control actions $u_t^{1:2}$ from the conditioning because they known given the strategy $\boldsymbol{g}$ and $m_t^1$. We prove \eqref{eq_lem_pi_cost_2} using the same arguments as above.
\end{proof}

Next, we prove that the state $S_t^2$ is sufficient for input-output mapping from the perspective of agent $2$.

\begin{lemma} \label{St_suff_pbp2}
    At each time $t$, the state $S_t^2 \in \mathcal{S}_t^2$ satisfies the following properties stated by Witsenhausen \cite{witsenhausen1976some}:
    
    1) There exist functions $\hat{f}^{{2}}_t(\cdot)$ and $\hat{h}^{{2}}_t(\cdot)$ for all $t=0,\dots,T-1$, such that
      \begin{align}
    	S^{{2}}_{t+1} &= \hat{f}^{{2}}_{t}(S^{{2}}_t, \Gamma_t^{[2,1]}, U_t^{{2}}, W_{t}, V_{t+1}^{1:2}),
    	\label{eq_St_pbp2_1} \\
    	Z^2_{t+1} &= \hat{h}^2_{t+1}( S^2_t,\Gamma_{t}^{[2,1]}, U_{t}^2, W_t, V_t^{1:2} ).
    	\label{eq_St_pbp2_2}
      \end{align}
      
    2) There exist functions $\hat{c}^{{2}}_t(\cdot)$, such that
        \begin{align}
    	c_t(X_t,U_t^{1:2}) &= \hat{c}^{{2}}_t({S}^{{2}}_t,\Gamma_t^{[2,1]},U^{{2}}_t), \quad t = 0,\dots,T. \label{eq_St_pbp2_3}
      \end{align}
\end{lemma}

\begin{proof}
To prove these results, we expand the LHS in each of \eqref{eq_St_pbp2_1}-\eqref{eq_St_pbp2_3} to write them in terms of the variables in the RHS, and construct appropriate functions $\hat{f}_t^1(\cdot)$, $\hat{h}_t^1(\cdot)$, and $\hat{c}_t^1(\cdot)$ for each time $t$.
\end{proof}

Lemma \ref{St_suff_pbp2} allows us to construct a centralized stochastic control problem for agent $2$ with state $S_t^2$ that evolves using \eqref{eq_St_pbp2_1}, control action $(\Gamma_t^{[2,1]},U_t^2)$, observation $Z_t^2$ given by \eqref{eq_St_pbp2_2}, and cost $\hat{c}_t^2(S_t^2, \Gamma_t^{[2,1]},U_t^2)$ at each time $t=0,\dots,T$. Furthermore, the performance criterion can be written as a function of the prescription strategy $\boldsymbol{\psi}^2$ and control strategy $\boldsymbol{g}^2$ as $\mathcal{J}^2(\boldsymbol{\psi}^2, \boldsymbol{g}^2) = \mathbb{E}^{\boldsymbol{g}} \left[\sum_{t=0}^T\hat{c}_t^2(S_t^2, \Gamma_t^{[2,1]},U_t^2)\right]$. 

\begin{problem} \label{problem_3}
The problem for agent $2$ is
    $\inf_{\boldsymbol{\psi}^2,\boldsymbol{g}^{2}} \mathcal{J}^1(\boldsymbol{\psi}^2,\boldsymbol{g}^{2})$,
given the probability distributions of the primitive random variables $\{X_0,W_{t},V_{t}^{1:2}:t = 0,\dots,T\}$, and the functions $\big\{\hat{c}^2_t,\hat{f}^2_t,\hat{h}^2_t:t=0,\dots,T\big\}$.
\end{problem}

In Problem \ref{problem_3}, at each time $t$, the component $\Pi_t^2$ of the state $S_t^2$ is completely observed by agent $2$. Furthermore, the unobserved component $\big\{X_t,L_t^{[1,2]}\big\}$ can be estimated by agent $2$ using the probability distribution $\Pi_t^2$. This yields the following structural result for agent $2$ in Problem \ref{problem_3}.

\begin{theorem} \label{pbp_struct_result_2}
    For agent $2$ in Problem \ref{problem_3}, without loss of optimality, we can restrict attention  to prescription strategies $\boldsymbol{\psi}^{*2}$ and control strategies $\boldsymbol{g}^{*2}$ with the structural form
    \begin{gather} \label{eq_pbp_struct_result_2_1}
        \Gamma_t^{[2,1]} = \psi_t^{*[2,1]}\big(\Pi_t^2\big), \quad t = 0,\dots,T, \\
        U_t^{2} = g_t^{*2}\big(\Pi_t^2\big), \quad t = 0,\dots,T. \label{eq_pbp_struct_result_2_2}
    \end{gather}
\end{theorem}

\begin{proof}
This result follows from standard arguments for partially observed Markov decision processes \cite[page 79]{varaiya_book}.
\end{proof}

Recall that using Lemmas \ref{lem_psi_g_relation} and \ref{lem_psi_g_relation_inv}, given a prescription strategy $\boldsymbol{\psi}^{*2}$ of the form in Theorem \ref{pbp_struct_result_2}, we can construct a corresponding control strategy $\boldsymbol{g}^{*1}$ for agent $1$ as $g_t^{*1}(\Pi_t^{1:2}) := \psi_t^{*[2,1]}(\Pi_t^2)(\Pi_t^1)$, for all $t =0,\dots,T$. Then, $g_t^{*1}$ and $\psi_t^{*[2,1]}$ yield the same control action $U_t^1$ at each time $t=0,\dots,T$. 
Thus, we can derive an optimal team strategy $\boldsymbol{g}^*$ for Problem \ref{problem_1} with the structural form
\begin{align}
    U_t^1 &= g_t^{*1}(\Pi_t^1,\Pi_t^2), \quad  t =0,\dots,T, \label{eq_corr_1}\\
    U_t^2 &= g_t^{*2}(\Pi_t^2), \quad \quad \; \;  t=0,\dots,T. \label{eq_corr_2}
\end{align}
To this end, we denote the set of feasible team strategies consistent with \eqref{eq_corr_1} and \eqref{eq_corr_2} by $\mathcal{G}''$.

\section{The Dynamic Program}

In this section, we present a DP to solve Problem \ref{problem_1} using the information states.
Recall that at each $t=0,\dots,T$, the memory $M_t^2$ and subsequently, the information state $\Pi_t^2$ are available to agent $1$. We extend the memory of agent $1$ at the final time step $T$ to also include $U_T^2$. Any team strategy $\boldsymbol{g} \in \mathcal{G}''$ can be implemented using the extended memory $\{\Pi_T^1, \Pi_T^2, U_T^2\}$ for agent $1$, by discarding $U_T^2$. Next, we show that the every strategy using the extended memory can also be implemented using only $\{\Pi_T^1, \Pi_T^2\}$.

\begin{lemma} \label{lem_extended}
Let $\bar{g}_T^1:\mathcal{P}_T^1\times \mathcal{P}_T^2 \times \mathcal{U}_T^2 \to \mathcal{U}_T^1$ be an extended control law for time $T$. Then, we can construct a control law $g_T^1:\mathcal{P}_T^1 \times \mathcal{P}_T^2 \to \mathcal{U}_T^1$ such that
\begin{gather}
    U_T^1 = \bar{g}_T^1(\Pi_T^1,\Pi_T^2,U_T^2) = g_T^1(\Pi_T^1,\Pi_T^2).
\end{gather}
\end{lemma}

\begin{proof}
The proof follows by substituting the relation $U_T^2 = g_T^2(\Pi_T^2)$ into the extended control law, and constructing $g_T^1(\cdot)$ as $g_T^1(\Pi_T^1,\Pi_T^2) := \bar{g}_T^1\big(\Pi_T^1,\Pi_T^2,g_T^2(\Pi_T^2)\big)$.
\end{proof}

Lemma \ref{lem_extended} establishes that we can equivalently select either $\bar{g}_T^1$ or $g_T^1$ at time $T$, because they yield the same control action $U_T^1$. To this end, we simply denote the control law of agent $1$ at time $T$ by $g_T^1$, even with the extended memory.

\subsection{The Value Functions}

In this subsection, we construct the value functions and corresponding control laws for our DP.
Let $u_t^k$ and $\pi_t^k$ be the realizations of the random variable $U_t^k$ and information state $\Pi_t^k$ for each $k \in \{1,2\}$, for each $t=0,\dots,T$. We recursively define two value functions at time $T$ as
\begin{align}
    &J^1_T(\pi_T^1, u_T^2) := \inf_{u_T^1 \in \mathcal{U}_T^1} \Tilde{c}^1_T(\pi^1_T,u_T^{1:2}), \label{value_T_1} \\
    &J^2_T(\pi_T^2) := \inf_{u_T^2 \in \mathcal{U}_T^2} \mathbb{E}^{(\boldsymbol{g}, \boldsymbol{\psi}^2)}\big[J^1_T(\Pi_T^1,u_T^2)~|~\pi_T^2, u_T^2\big]. \label{value_T_2}
\end{align}
The control law for agent $1$ at time $T$ is $u_T^{*1} = g_T^{*1}(\pi_T^1,u_T^2)$, i.e., the $\arg\inf$ in the RHS of \eqref{value_T_1}. The control law for agent $2$ is $u_T^{*2} = g_T^{*2}(\pi_T^2)$, i.e., the $\arg\inf$ in the RHS of \eqref{value_T_2}. 

Next, at each time $t=T-1, \dots, 0$, we recursively define
\begin{align}
    &J_t(\pi_t^2) := \inf_{u_t^2 \in \mathcal{U}_t^2, \gamma_t^{[2,1]} \in \mathcal{F}_t^{[2,1]}} \tilde{c}_t\big(\pi_t^2,\gamma_t^{[2,1]},u_t^{1}\big) \nonumber \\
    &\quad \quad \quad + \mathbb{E}^{(\boldsymbol{g}, \boldsymbol{\psi}^2)}\Big[{J}_{t+1}\big(\Pi_{t+1}^2 \big) ~|~\pi_t^{2}, \gamma_t^{[2,1]}, u_t^{2}\Big], \label{value_time_t}
\end{align}
where at time $T-1$, by convention $J_T(\Pi_T^2) = J_T^2(\Pi_T^2)$.
The prescription law at time $t$ is $\gamma_t^{*[2,1]} = \psi_t^{*[2,1]}(\pi_t^{2})$ and the control law of agent $2$ is $u_t^{*2} = g_t^{*2}(\pi_t^2)$, i.e., the $\arg\inf$ in the RHS of \eqref{value_time_t}. 
The value functions \eqref{value_T_1}-\eqref{value_time_t} and corresponding control laws form a DP for the team.

\subsection{Optimality of the Dynamic Program}

In this subsection, we prove the optimality of our DP, starting with time $T$. Let $\boldsymbol{g}_t=(g_t^1,g_t^2)$ and $\boldsymbol{g}_{0:t} = (\boldsymbol{g}_0,\dots,\boldsymbol{g}_t)$. Furthermore, let
    $\mathcal{J}_t(\boldsymbol{g}) := \mathbb{E}^{(\boldsymbol{g}, \boldsymbol{\psi}^2)}\left[\sum_{\ell=t}^T c_\ell(X_\ell,U_\ell^{1:2})\right].$
Next, we show that the control law $g_T^{*1}$ is optimal for agent $1$ at time $T$.

\begin{lemma} \label{lem_12}
1) The value function $J_T^1$ in \eqref{value_T_1} is such that
\begin{align}
    \mathcal{J}_T(\boldsymbol{g}) 
    \geq &\mathbb{E}^{(\boldsymbol{g}, \boldsymbol{\psi}^2)}[J_T^1(\Pi^1_T,U_T^2)], \quad \forall \boldsymbol{g} \in \mathcal{G}''.
\end{align}

2) The corresponding control law $g_T^{*1}$ is such that
\begin{gather}
    \mathcal{J}_T(\boldsymbol{g}_{0:T-1},g_T^{*1},g_T^2) = \mathbb{E}^{(\boldsymbol{g}, \boldsymbol{\psi}^2)}[J_T^1(\Pi^1_T,U_T^2)].
\end{gather}
\end{lemma}
\begin{proof}
1) Using the extended memory of agent $1$ at time $T$, $U_T^1 = g_T^1(\Pi_T^{1:2},U_T^{2})$. Using Lemma \ref{lem_pi_cost}, we write that
\begin{align*}
    \mathcal{J}_T(\boldsymbol{g}) 
    &= \mathbb{E}^{(\boldsymbol{g}, \boldsymbol{\psi}^2)}\big[\tilde{c}_T^1(\Pi_T^1,U_T^{1:2})\big] 
    \geq \mathbb{E}^{(\boldsymbol{g}, \boldsymbol{\psi}^2)}\big[J_T^1(\Pi_T^1,U_T^2)\big], \label{lem_12_1}
\end{align*}
where, in the inequality, we used the definition of $J_T^1$ in \eqref{value_T_1}.

2) We substitute $U_T^1 = g_T^{*1}(\Pi_T^1,U_T^2)$ in the expansion of $\mathcal{J}_T$, to write that
\begin{align}
     \mathcal{J}_T(\boldsymbol{g}_{0:T-1},g_T^{*1},g_T^2) &=\mathbb{E}^{(\boldsymbol{g}, \boldsymbol{\psi}^2)}\big[\tilde{c}^1_T(\Pi^1_T, g_T^{*1}(\Pi_T^1,U_T^2), U_T^{2})\big] \nonumber \\
     &= \mathbb{E}^{(\boldsymbol{g}, \boldsymbol{\psi}^2)}[J_T^1(\Pi_T,U_T^2)],
\end{align}
where, in the second equality, we expand the expectation in the LHS, substitute the definitions of $g_T^{*1}$ and $J_T^1$ from \eqref{value_T_1}, and note that this yields the expectation in the RHS.
\end{proof}

Next, we show that, given the control law $g_T^{*1}$ for agent $1$, the control law $g_T^{*2}$ is optimal for agent $2$ at time $T$. 


\begin{lemma} \label{lem_13}
    1) The value function $J_T^2$ in \eqref{value_T_2} is such that
\begin{align}
    \mathcal{J}_T(\boldsymbol{g}) \geq \mathbb{E}^{(\boldsymbol{g}, \boldsymbol{\psi}^2)}[J_T^2(\Pi^2_T)], \quad \forall \boldsymbol{g} \in \mathcal{G}''.
\end{align}

2) The corresponding control law $g_T^{*2}$ is such that
\begin{gather}
    \mathcal{J}_T(\boldsymbol{g}_{0:T-1},g_T^{*1},g_T^{*2}) = \mathbb{E}^{(\boldsymbol{g}, \boldsymbol{\psi}^2)}[J_T^2(\Pi^2_T)].
\end{gather}
\end{lemma}

\begin{proof}
1) Agent $2$ at time $T$ generates her action $U_T^2$ as a function of $\Pi^2_T$. Using Lemma \ref{lem_12}, for all $\boldsymbol{g} \in \mathcal{G}''$,
\begin{align}
    \mathcal{J}_T(\boldsymbol{g}) &\geq \mathbb{E}^{(\boldsymbol{g}, \boldsymbol{\psi}^2)}[J_T^1(\Pi^1_T,U_T^2)] \nonumber \\
    &= \mathbb{E}^{(\boldsymbol{g}, \boldsymbol{\psi}^2)}\big[\mathbb{E}^{(\boldsymbol{g}, \boldsymbol{\psi}^2)}[J_T^1(\Pi^1_T,U_T^2)|\Pi_T^2,U_T^2]\big] \nonumber \\
    &\geq \mathbb{E}^{(\boldsymbol{g}, \boldsymbol{\psi}^2)}\big[J_T^2(\Pi^2_T)\big], \label{lem_13_1}
\end{align}
where, in the equality, we use the law of iterated expectations, and in the second inequality, we use the definition of $J_T^2$ from \eqref{value_T_2}.

2) Starting with the equality in \eqref{lem_13_1}, we substitute $U_T^1 = g_T^{*2}(\Pi_T^1)$ to write that
\begin{align}
    & \mathcal{J}_T(\boldsymbol{g}_{0:T-1},g_T^{*1},g_T^{*2}) \nonumber \\
    &= \mathbb{E}^{(\boldsymbol{g}, \boldsymbol{\psi}^2)}\big[\mathbb{E}^{(\boldsymbol{g}, \boldsymbol{\psi}^2)}[J_T^1(\Pi^1_T,g_T^{*2}(\Pi_T^2)~|~\Pi_T^2]\big] \nonumber \\
    &=\mathbb{E}^{(\boldsymbol{g}, \boldsymbol{\psi}^2)}\big[J_T^2(\Pi^2_T)\big],
\end{align}
where, in the second equality, we expand the expectation in the LHS, substitute the definitions of $g^{*2}_T$ and $J_T^2$ from \eqref{value_T_2}, and note that this yields the expectation in the RHS.
\end{proof}

Next, we show that the laws $(\psi_t^{*[2,1]},g_t^{*2})$ are optimal for each $t=0,\dots,T-1$.

    %

\begin{theorem} \label{lemma_15}
For any $t=0,\dots,T-1$:

1) The value function $J_t$ in \eqref{value_time_t} is such that
\begin{align}
    \mathcal{J}_t(\boldsymbol{g})
    \geq &\mathbb{E}^{(\boldsymbol{g}, \boldsymbol{\psi}^2)}\big[J_t(\Pi^{2}_t)\big], \quad \forall \boldsymbol{g} \in \mathcal{G}''.
\end{align}

2) The corresponding laws $(\psi_t^{*[2,1]}, g_t^{*2})$ are such that
\begin{gather} \label{lem_15_condition_2}
    \mathcal{J}_t(\boldsymbol{g}_{0:t-1},g_{t}^{*1},g_t^{*2}, ,\boldsymbol{g}_{t+1:T}^{*}) = \mathbb{E}^{(\boldsymbol{g}, \boldsymbol{\psi}^2)}[J_t(\Pi^{2}_t)],
\end{gather}
where $g_t^{*1}$ is derived from $\psi_t^{*[2,1]}$ using Lemma \ref{lem_psi_g_relation_inv}.
\end{theorem}

\begin{proof}
Note the DP for time steps $t = 0, \dots, T-1$ is the same as a centralized DP for Problem \ref{problem_3}. The optimality of such a DP can be proven using mathematical induction starting with time $T-1$ in a manner similar to \cite{malikopoulos2021optimal, nayyar2019common}. 
\end{proof}

\begin{remark}
At time $T$, our DP has two sub-steps, each with a different value function. These sub-steps take advantage of the nested information structure to directly compute the control laws at time $T$, which involves solving two optimization problems with respect to control actions. This is simpler than solving an optimization problem involving a prescription of agent $2$ for agent $1$, as in time steps $t=0,\dots,T-1$. Thus, our DP presents a simpler solution for the final time step.
\end{remark}

\section{Discussion and Conclusions}

In this paper, we introduced a dynamic team of two agents with a nested information structure and derived structural results for the optimal control strategies. Our derivation utilized a combination of the person-by-person and prescription approaches to arrive at a distinct structural form that cannot be achieved by either of the techniques alone.
We also presented a DP that can be used to derive the optimal control strategies for a finite time horizon. Our DP utilized the nested information structure to simplify the computation of optimal control laws for the team at the final time step.
Note that our results can be extended to teams of $n \in \mathbb{N}$ agents with nested information, 
by iteratively applying the person-by-person and the prescription approach. While our results do not yield a time invariant domain for optimal control strategies, their advantage is that they only require tracking probability distributions over finite valued supports. Thus, in comparison to related results in \cite{nayyar2011structure, Mahajan2015}, it may be easier to derive approximate strategies using our results. 
Furthermore, there may be systems with specific dynamics where we can divide the DP into multiple sub-steps at each time step. Our ongoing work seeks to extend our results to more general information structures and to investigate decentralized minimax control problems using these techniques like those reported in \cite{gagrani2017decentralized}.

\bibliographystyle{ieeetr}
\bibliography{References}

\begin{thebibliography}{10}

\bibitem{radner1962team}
R.~Radner, ``Team decision problems,'' {\em The Annals of Mathematical
  Statistics}, vol.~33, no.~3, pp.~857--881, 1962.

\bibitem{malikopoulos2021optimal}
A.~A. Malikopoulos, L.~Beaver, and I.~V. Chremos, ``Optimal time trajectory and
  coordination for connected and automated vehicles,'' {\em Automatica},
  vol.~125, p.~109469, 2021.

\bibitem{Dave2020SocialMedia}
A.~Dave, I.~V. Chremos, and A.~A. Malikopoulos, ``{Social Media and Misleading
  Information in a Democracy: A Mechanism Design Approach},'' {\em IEEE
  Transactions on Automatic Control}, 2022 (in press).

\bibitem{Beaver2020AnFlockingb}
L.~E. Beaver and A.~A. Malikopoulos, ``{An Overview on Optimal Flocking},''
  {\em Annual Reviews in Control}, vol.~51, pp.~88--99, 2021.

\bibitem{varaiya_book}
P.~R. Kumar and P.~P. Varaiya, {\em Stochastic Systems: Estimation,
  Identification, and Adaptive Control}.
\newblock Englewood Cliffs, NJ: Prentice-Hall, 1986.

\bibitem{2}
Y.-C. Ho and K.-C. Chu, ``{Team decision theory and information structures in
  optimal control problems--Part I},'' {\em IEEE Transactions on Automatic
  Control}, vol.~17, no.~1, pp.~15--22, 1972.

\bibitem{lessard2013structural}
L.~Lessard and A.~Nayyar, ``Structural results and explicit solution for
  two-player lqg systems on a finite time horizon,'' in {\em 52nd IEEE
  Conference on Decision and Control}, pp.~6542--6549, IEEE, 2013.

\bibitem{nayyar2015structural}
A.~Nayyar and L.~Lessard, ``Structural results for partially nested lqg systems
  over graphs,'' in {\em 2015 American Control Conference (ACC)},
  pp.~5457--5464, IEEE, 2015.

\bibitem{yuksel2009stochastic}
S.~Yuksel, ``Stochastic nestedness and the belief sharing information
  pattern,'' {\em IEEE Transactions on Automatic Control}, vol.~54, no.~12,
  pp.~2773--2786, 2009.

\bibitem{Aditya2019}
A.~Dave and A.~A. Malikopoulos, ``Decentralized stochastic control in partially
  nested information structures,'' in {\em 8th IFAC Workshop on Distributed
  Estimation and Control in Networked Systems}, 2019.

\bibitem{17}
A.~Nayyar, A.~Mahajan, and D.~Teneketzis, ``Decentralized stochastic control
  with partial history sharing: A common information approach,'' {\em IEEE
  Transactions on Automatic Control}, vol.~58, no.~7, pp.~1644--1658, 2013.

\bibitem{19}
A.~Nayyar, T.~Ba{\c s}ar, D.~Teneketzis, and V.~V. Veeravalli, ``{Optimal
  Strategies for Communication and Remote Estimation With an Energy Harvesting
  Sensor},'' {\em IEEE Transactions on Automatic Control}, vol.~58, no.~9,
  pp.~2246--2260, 2013.

\bibitem{dave2020structural}
A.~Dave and A.~A. Malikopoulos, ``Structural results for decentralized
  stochastic control with a word-of-mouth communication,'' in {\em 2020
  American Control Conference (ACC)}, pp.~2796--2801, IEEE, 2020.

\bibitem{malikopoulos2021team}
A.~A. Malikopoulos, ``On team decision problems with nonclassical information
  structures,'' {\em arXiv:2101.10992}, 2021 (in review).

\bibitem{nayyar2011structure}
A.~Nayyar and D.~Teneketzis, ``On the structure of real-time encoding and
  decoding functions in a multiterminal communication system,'' {\em IEEE
  transactions on information theory}, vol.~57, no.~9, pp.~6196--6214, 2011.

\bibitem{mahbub2021_platoonMixed}
A.~M.~I. Mahbub and A.~A. Malikopoulos, ``{A Platoon Formation Framework in a
  Mixed Traffic Environment},'' {\em IEEE Control Systems Letters (LCSS)},
  vol.~6, pp.~1370--1375, 2021.

\bibitem{mahajan2015algorithmic}
A.~Mahajan and S.~Tatikonda, ``An algorithmic approach to identify irrelevant
  information in sequential teams,'' {\em Automatica}, vol.~61, pp.~178--191,
  2015.

\bibitem{nayyar2019common}
A.~Nayyar and D.~Teneketzis, ``Common knowledge and sequential team problems,''
  {\em IEEE Transactions on Automatic Control}, vol.~64, no.~12,
  pp.~5108--5115, 2019.

\bibitem{yuksel2020universal}
S.~Yuksel, ``A universal dynamic program and refined existence results for
  decentralized stochastic control,'' {\em SIAM Journal on Control and
  Optimization}, vol.~58, no.~5, pp.~2711--2739, 2020.

\bibitem{mahajan2012}
A.~Mahajan, N.~C. Martins, M.~C. Rotkowitz, and S.~Y{\"u}ksel, ``Information
  structures in optimal decentralized control,'' in {\em 2012 IEEE 51st IEEE
  Conference on Decision and Control (CDC)}, pp.~1291--1306, IEEE, 2012.

\bibitem{charalambous2016decentralized}
C.~D. Charalambous, ``Decentralized optimality conditions of stochastic
  differential decision problems via girsanov's measure transformation,'' {\em
  Mathematics of Control, Signals, and Systems}, vol.~28, no.~3, pp.~1--55,
  2016.

\bibitem{gagrani2017decentralized}
M.~Gagrani and A.~Nayyar, ``Decentralized minimax control problems with partial
  history sharing,'' in {\em 2017 American Control Conference (ACC)},
  pp.~3373--3379, IEEE, 2017.

\bibitem{mahajan2013controlsharing}
A.~Mahajan, ``Optimal decentralized control of coupled subsystems with control
  sharing,'' {\em IEEE Transactions on Automatic Control}, vol.~58, no.~9,
  pp.~2377--2382, 2013.

\bibitem{subramanian2019approximate}
J.~Subramanian and A.~Mahajan, ``Approximate information state for partially
  observed systems,'' in {\em 2019 IEEE 58th Conference on Decision and Control
  (CDC)}, pp.~1629--1636, IEEE, 2019.

\bibitem{witsenhausen1976some}
H.~Witsenhausen, ``Some remarks on the concept of state,'' in {\em Directions
  in Large-Scale Systems}, pp.~69--75, Springer, 1976.

\bibitem{dave2021dynamic}
A.~Dave and A.~A. Malikopoulos, ``A dynamic program for a team of two agents
  with nested information,'' {\em preprint, arXiv:2103.10028}, 2021.

\bibitem{Mahajan2015}
J.~{Arabneydi} and A.~{Mahajan}, ``Team optimal control of coupled subsystems
  with mean-field sharing,'' in {\em 53rd IEEE Conference on Decision and
  Control}, pp.~1669--1674, Dec 2014.

\end{thebibliography}

\section*{Appendix A - Proof of Lemma \ref{pi_rel}}

Let $x_t$, $\gamma_t^{[2,1]}$, ${s}_t^{2}$, $m_t^k$, $l_t^{[1,2]}$, and $\pi_t^k$ be the realizations of the random variables $X_t$, $\Gamma_t^{[2,1]}$, ${S}_t^{2}$, $M_t^k$, $L_t^{[1,2]}$, and the conditional probability distribution $\Pi_t^k$, respectively, for all $k \in \{1,2\}$, for all $t=0,\dots,T$. Using \eqref{pi_1_with_gamma},
\begin{align}
    \pi_t^1(x_t) &= \mathbb{P}^{(\boldsymbol{g}, \boldsymbol{\psi}^2)} (X_t = x_t ~|~ m_t^1, \gamma_{0:t-1}^{[2,1]}) \nonumber \\
    &= \mathbb{P}^{(\boldsymbol{g}, \boldsymbol{\psi}^2)} \big(X_t = x_t ~|~ m_t^2, l_t^{[1,2]}, \gamma_{0:t-1}^{[2,1]}\big).
    \label{eq_lem_1_m1}
\end{align}
We use Bayes' rule to expand \eqref{eq_lem_1_m1} as
\begin{multline}
    \mathbb{P}^{(\boldsymbol{g}, \boldsymbol{\psi}^2)} \big(X_t = x_t ~|~ m_t^2, l_t^{[1,2]}, \gamma_{0:t-1}^{[2,1]}\big)  \\
    = \dfrac{\pi_t^2({s}_t^2)} {\mathbb{P}^{(\boldsymbol{g}, \boldsymbol{\psi}^2)} \big(L_t^{[1,2]} = l_t^{[1,2]} ~|~ m_t^2, \gamma_{0:t-1}^{[2,1]}\big)}, \label{eq_lem_1_0}
\end{multline}
where $\pi_t^2(s_t^2) = \mathbb{P}^{(\boldsymbol{g}, \boldsymbol{\psi}^2)} (S_t^2 = s_t^2~|~ m_t^2, \gamma_{0:t-1}^{[2,1]})$ and ${s}_t^{2} = \big\{x_t,l_t^{[1,2]}\big\}$.
We expand the denominator in \eqref{eq_lem_1_0} using the law of total probability, as
\begin{align}
    \mathbb{P}^{(\boldsymbol{g}, \boldsymbol{\psi}^2)} &\big(L_t^{[1,2]} = l_t^{[1,2]} ~|~ m_t^2, \gamma_{0:t-1}^{[2,1]}\big) \nonumber \\
    = \sum_{\tilde{s}^2_t \in \Tilde{\mathcal{S}}^2_t} &\mathbb{P}^{(\boldsymbol{g}, \boldsymbol{\psi}^2)} \big(\tilde{s}^2_t, l_t^{[1,2]} ~| ~ m_t^2, \gamma_{0:t-1}^{[2,1]}\big), \nonumber \\
    = \sum_{\tilde{s}^2_t \in \Tilde{\mathcal{S}}^2_t} &\mathbb{P}^{(\boldsymbol{g}, \boldsymbol{\psi}^2)} \big(L_t^{[1,2]} = l_t^{[1,2]} ~| ~ \tilde{s}^2_t, m_t^2, \gamma_{0:t-1}^{[2,1]}\big) \nonumber \\
    &\cdot \mathbb{P}^{(\boldsymbol{g}, \boldsymbol{\psi}^2)} \big(S_t^2 = \tilde{s}^2_t ~| ~ m_t^2, \gamma_{0:t-1}^{[2,1]}\big), \nonumber \\
    = \sum_{\tilde{s}^2_t \in \Tilde{\mathcal{S}}^2_t} &\mathbb{I}\big({l}_t^{[1,2]} \in \tilde{s}_t^{2}\big)\cdot \pi_t^2(\tilde{s}_t^2), \label{eq_lem_1_2}
\end{align}
where, in the second equality, we use Bayes' law, $\pi_t^2(\tilde{s}_t^2) = \mathbb{P}^{(\boldsymbol{g}, \boldsymbol{\psi}^2)}\big({S}_t^2 = \tilde{s}_t^2 ~|~ m_t^2, \gamma_{0:t-1}^{[2,1]}\big)$, and $\mathbb{I}(\cdot)$ is the indicator function.
The result holds by substituting \eqref{eq_lem_1_0}-\eqref{eq_lem_1_2} into \eqref{eq_lem_1_m1} and constructing an appropriate function ${e}_t(\cdot)$.

\section*{Appendix B - Proof of Lemma \ref{pi_2_1}}

Let $x_t$, $u_t^k$, ${s}_t^{2}$, $m_t^k$, $l_t^{[1,2]}$, $\gamma_t^{[2,1]}$, and $\pi_t^k$ be the realizations of the random variables $X_t$, $U_t^k$, ${S}_t^{2}$, $M_t^k$, $L_t^{[1,2]}$, $\Gamma_t^{[2,1]}$, and the conditional probability distribution $\Pi_t^k$, respectively. Let $(\boldsymbol{g}, \boldsymbol{\psi}^2)$ be the strategy profile for the team.
Then, by definition \eqref{pi_2_def},
\begin{gather} 
    \pi^{{2}}_{t+1}(s_{t+1}^2) = \mathbb{P}^{(\boldsymbol{g},\boldsymbol{\psi}^{{2}})}\big(S_{t+1}^2 = s_{t+1}^2 ~|~ m^{{2}}_{t+1},\gamma_{0:t}^{{[2,1]}}\big) \nonumber \\
    = \mathbb{P}^{(\boldsymbol{g},\boldsymbol{\psi}^{{2}})}\big(S_{t+1}^2 = s_{t+1}^2 ~|~ \pi_t^2, m^{{2}}_{t+1}, \gamma_{0:t}^{{[2,1]}}\big), \label{lem_6_1}
\end{gather}
where we can add the distribution $\pi_t^2$ to the conditioning because it is known given $(m^{{2}}_{t}, \gamma_{0:t-1}^{{[2,1]}})$ and $(\boldsymbol{g},\boldsymbol{\psi}^{{2}})$. We know that $s_{t+1}^2 = \big\{x_{t+1},l^{[1,2]}_{t+1}\big\}$, where $l^{[1,2]}_{t+1} = l^{[1,2]}_t \cup \{y_{t+1}^1,u_t^1\}$. Then, we can use \eqref{st_eq}, \eqref{ob_eq}, and Lemma \ref{pi_rel} to construct an appropriate function $\phi_t(\cdot)$ such that
\begin{gather} \label{lem_6_2}
    s_{t+1}^2 = \phi_t\Big(s^2_t, \pi_t^2, \gamma_t^{[2,1]},u_t^2, w_t, v_{t+1}^{1:2}\Big).
\end{gather}
We can substitute \eqref{lem_6_2} into \eqref{lem_6_1} to state that
\begin{gather}
    \pi^{{2}}_{t+1}(s_{t+1}^2) = \sum_{s_{t}^2, w_t, v_{t+1}^{1:2}} \mathbb{I}\Big(\phi_t\big(s_{t}^2,\pi_t^2, \gamma_t^{[2,1]}, u_t^2, w_t, v_{t+1}^{1:2}\big)  \nonumber \\
    = s_{t+1}^2\Big) \cdot \mathbb{P}(w_t, v_{t+1}^{1:2}) \cdot \mathbb{P}^{(\boldsymbol{g},\boldsymbol{\psi}^{{2}})}\big(S_{t}^2 = s_{t}^2| \pi_t^2, m^{{2}}_{t+1}, \gamma_{0:t}^{{[2,1]}}\big), \label{lem_6_3}
\end{gather}
where $m_{t+1}^2 = \{m_t^2, z_{t+1}^2\}$. Then, the last term in \eqref{lem_6_3} can be written as
\begin{multline}
    \mathbb{P}^{(\boldsymbol{g},\boldsymbol{\psi}^{{2}})}\Big(S_{t}^2 = s_{t}^2 ~|~ \pi_t^2, m^{{2}}_{t}, z_{t+1}^2, \gamma_{0:t}^{{[2,1]}}\Big)  \\
    = \dfrac{\left[\splitfrac{\mathbb{P}^{(\boldsymbol{g},\boldsymbol{\psi}^{{2}})}\big(Z^2_{t+1}=z^2_{t+1}~|~s_{t}^2,\pi_t^2,m^2_{t},\gamma_{0:t}^{{[2,1]}}\big)}
    {\cdot\mathbb{P}^{(\boldsymbol{g},\boldsymbol{\psi}^{{2}})}\big(s_{t}^2,m^2_{t},\pi_t^2,\gamma_{0:t}^{[2,1]}\big)}\right]}
    {\left[\sum_{\tilde{s}^2_t}{\splitfrac{\mathbb{P}^{(\boldsymbol{g},\boldsymbol{\psi}^{{2}})}\big(Z^2_{t+1}=z^2_{t+1}~|~\tilde{s}^2_t, \pi_t^2 ,m^2_{t},\gamma_{0:t}^{{[2,1]}}\big)}
    {\cdot \mathbb{P}^{(\boldsymbol{g},\boldsymbol{\psi}^{{2}})}\big(\tilde{s}^2_t,m^2_{t},\pi_t^2, \gamma_{0:t}^{[2,1]}\big)}}\right]}. \label{lem_6_4}
\end{multline}
Note that using Bayes' law, it holds that
\begin{multline}
   \mathbb{P}^{(\boldsymbol{g}, \boldsymbol{\psi}^2)}\big(s^2_t,m^2_{t},\pi_t^2,\gamma_{0:t}^{[2,1]}\big) \nonumber \\
    = \mathbb{P}^{(\boldsymbol{g}, \boldsymbol{\psi}^2)}\big(S_t^2 = s_t^2~|~ m^2_{t},\gamma_{0:t-1}^{[2,1]}\big) \cdot \mathbb{P}^{(\boldsymbol{g}, \boldsymbol{\psi}^2)}\big(\pi_t^2,m^2_{t},\gamma_{0:t}^{[2,1]}\big),
\end{multline}
where $\pi_t^2, \gamma_{t}^{[2,1]}$ are known given $(m_t^2, \gamma_{0:t-1}^{[2,1]})$ and $(\boldsymbol{g},\boldsymbol{\psi}^{{2}})$. Substituting into \eqref{lem_6_4},
\begin{multline}
    \mathbb{P}^{(\boldsymbol{g},\boldsymbol{\psi}^{{2}})}\Big(S_{t}^2 = s_{t}^2 ~|~ \pi_t^2, m^{{2}}_{t}, z_{t+1}^2, \gamma_{0:t}^{{[2,1]}}\Big) \\
    = \dfrac{\mathbb{P}^{(\boldsymbol{g},\boldsymbol{\psi}^{{2}})}\big(Z^2_{t+1}=z^2_{t+1}|s_{t}^2,\pi_t^2,m^2_{t},\gamma_{0:t}^{{[2,1]}}\big)
    \cdot \pi_t^2(s_{t}^2)}
    {\left[\sum_{\tilde{s}^2_t}{\splitfrac{\mathbb{P}^{(\boldsymbol{g},\boldsymbol{\psi}^{{2}})}\big(Z^2_{t+1}=z^2_{t+1}|\tilde{s}_{t}^2, \pi_t^2 ,m^2_{t},\gamma_{0:t}^{{[2,1]}}\big)}
    {\cdot \pi_t^2(\tilde{s}^2_t)}}\right]}. \label{lem_6_5}
\end{multline}
Next, we can use \eqref{z_2}, in addition to \eqref{st_eq}, \eqref{ob_eq}, and Lemma \ref{pi_rel}, to construct an appropriate function $\hat{\phi}_t(\cdot)$ such that
\begin{gather} \label{proof_2_2}
    z_{t+1}^2 = \hat{\phi}_t\big(s_t^2,\pi_t^2,\gamma_t^{[2,1]}, u_t^2, w_t, v_{t+1}^{1:2}\big),
\end{gather}
which implies that
\begin{multline}
    \mathbb{P}^{(\boldsymbol{g},\boldsymbol{\psi}^{{2}})}(Z^2_{t+1}=z^2_{t+1}|s^2_t,\pi_t^2,m^2_{t},\gamma_{0:t}^{{[2,1]}}) = \sum_{w_t,v_{t+1}^{1:2}}\mathbb{I}\Big( \\
    \hat{\phi}_t(s^2_t,\pi_t^2,\gamma_t^{[2,1]}, u_t^2, w_t, v_{t+1}^{1:2}) = z_{t+1}^2\Big)
    \cdot \mathbb{P}(w_t, v_{t+1}^{1:2}\big). \label{lem_6_6}
\end{multline}
The first result holds by substituting \eqref{lem_6_5} and \eqref{lem_6_6} into \eqref{lem_6_3}.

Furthermore, for some Borel subset $P^2 \subseteq \mathcal{P}_t^2$,
\begin{multline}
    \mathbb{P}(\Pi_{t+1}^2 \in P^2~|~ m_t^2, \gamma^{[2,1]}_{0:t},\pi^2_{0:t})
    =\sum_{z^2_{t+1}}\mathbb{I}\Big[\tilde{f}_{t+1}^2(\pi_t^2,\gamma^{[2,1]}_{t}, \\
    u^2_t, z_{t+1}^2)
    \in P\Big]
    \cdot\mathbb{P}(Z^2_{t+1}=z^{{k}}_{t+1}~|~m_t^2,\gamma^{[2,1]}_{0:t},\pi_{0:t}^2)\label{proof_2_1},
\end{multline}
where $\mathbb{I}(\cdot)$ is the indicator function. Using \eqref{proof_2_2}, the second term in \eqref{proof_2_1} can be expanded as
\begin{multline}
    \mathbb{P}(Z^2_{t+1}=z^2_{t+1}~|~m_t^2,\gamma^{[2,1]}_{0:t},\pi_{0:t}^2) \\
    =\sum_{s_t^2, w_t, v_{t+1}^{1:2}}\mathbb{I}\Big[\tilde{\phi}_t(s_t^2,\pi_t^2,\gamma_t^{[2,1]}, u_t^2, w_t, v_{t+1}^{1:2}) = z_{t+1}^2 \Big]\cdot\mathbb{P}(V_{t+1}^{{1:2}} \\ =v_{t+1}^{{1:2}})
    \cdot \mathbb{P}(W_t = w_t)\cdot\mathbb{P}(S_t^2=s_t^2~|~m_t^2,\gamma_{0:t}^{[2,1]}, u_{0:t}^2,\pi_{0:t}^2) \\
    =\sum_{s_t^2, w_t, v_{t+1}^{1:2}}\mathbb{I}\Big[\tilde{\phi}_t(s_t^2,\pi_t^2,\gamma_t^{[2,1]}, u_t^2, w_t, v_{t+1}^{1:2}) = z_{t+1}^2 \Big]\cdot\mathbb{P}(V_{t+1}^{{1:2}} \\ =v_{t+1}^{{1:2}})
    \cdot \pi_t^2(s_t^2)
    \label{proof_2_1a}.
\end{multline}
Substituting \eqref{proof_2_1a} into \eqref{proof_2_1}, the proof is complete.



\end{document}